\documentclass[11pt,dvipsnames]{article}
\usepackage{amsmath,amscd}
\usepackage{amsthm}
\usepackage{amsfonts}
\usepackage{stmaryrd}
\usepackage{amssymb}
\usepackage[all]{xy}
\usepackage{comment}
\usepackage{graphics}
\usepackage{enumitem}
\usepackage{tikz}
\usetikzlibrary{cd}
\usepackage{amsmath,amsthm,amscd}
\usepackage{mathrsfs}
\newcommand\N{\mathbb{N}}
\newcommand\R{\mathbb{R}}

\newcommand\Q{\mathbb{Q}}

\newcommand\del{\partial}

\newcommand{\Diff}{\mathrm{Diff}}

\newcommand{\BDiff}{\mathrm{BDiff}}
\newcommand{\EDiff}{\mathrm{EDiff}}
\newcommand{\Ric}{\mathrm{Ric}}
\newcommand\pos{\mathsf{pos}}

\def\calR{\mathscr{R}}
\def\calM{\mathscr{M}}
\def\calC{\mathscr{C}}
\def\calQ{\mathcal{Q}}
\newcommand{\rh}{\hookrightarrow}
\newcommand{\Emb}{\mathrm{Emb}}
\newcommand\e{\mathsf{e}}
\newcommand\m{\mathsf{m}}
\theoremstyle{plain}

\newtheorem*{maintheorem}{Main Theorem}
\newtheorem{theorem}{Theorem}[section]
\newtheorem{lemma}[theorem]{Lemma}
\newtheorem{corollary}[theorem]{Corollary}
\newtheorem{proposition}[theorem]{Proposition}

\newtheorem{observation}[theorem]{Observation}

\theoremstyle{definition}
\newtheorem{definition}[theorem]{Definition}

\newtheorem*{Convention}{Convention}
\theoremstyle{remark}
\newtheorem{remark}[theorem]{Remark}

\theoremstyle{definition}

\usepackage{color}

\usepackage{xcolor}

\usepackage{soul}
\setstcolor{ForestGreen}

\usepackage{hyperref}
\hypersetup{
	colorlinks=true,
	linkcolor=blue,
	citecolor=blue,
	urlcolor=blue
}

\advance \topmargin by -\headheight
\advance \topmargin by -\headsep
     
\evensidemargin \oddsidemargin
\makeatletter
\newcommand{\subjclass}[2][1991]{%
	\let\@oldtitle\@title%
	\gdef\@title{\@oldtitle\footnotetext{#1 \emph{Mathematics subject classification.} #2}}%
}

\begin{document}

\title{On the topology of the moduli space of positive scalar curvature concordances}
\author{Boris Botvinnik and David J. Wraith}
  
\maketitle
\begin{abstract}
  Let $M$ be a manifold which admits a metric with positive scalar
  curvature (or a positive intermediate curvature in a suitable
  sense).  We study the moduli space $\calM^{\pos_*}_{\sqcup}(M\times
  I)_g$ of concordances of such metrics (with appropriate boundary
  conditions) which restrict to a given metric $g$ on $M \times \{0\}
  \cup\partial M \times I$. We show that
  $\pi_{4*}\calM^{\pos_*}_{\sqcup}(M \times I)_g \otimes \Q \neq 0$ in
  a stable range provided $\dim M$ is even. We obtain
  analogous results when positive scalar curvature is replaced by
  $k$-positive Ricci curvature for $k \ge 2$.
\end{abstract}
\section{Introduction}\label{intro}

We begin by providing a gentle introduction to the results in this paper, before going on to provide the full details.

Let $M$ be a compact manifold which supports a positive scalar curvature metric $g$. If $M$ is closed, a positive scalar curvature concordance (based at $g$) is a positive scalar curvature metric on $M \times [0,1]$ which restricts to $g$ at $M \times \{0\}$, and infinitessimally agrees with a product metric at both boundary components. On the other hand, if $M$ has non-empty boundary, then a positive scalar curvature concordance must agree infinitessimally with the product $g+dr^2$ on $(M \times \{0\}) \cup (\partial M \times [0,1])$, i.e. on `three sides' of the boundary, as well as having an infinitessimal product structure at $M \times \{1\}$. In either situation, the notion of concordance is classical in positive scalar curvature geometry. 

Fixing $g$, one can we consider the {\it space} of positive
scalar curvature concordances on $M \times I$ based at $g$ (equipped
with the smooth topology). Furthermore, there is a natural group of
diffeomorphisms which acts on the space of concordances: this is the
group of pseudoisotopies. Thus one can form the quotient space, which
by definition is the {\it moduli space of positive scalar curvature
  concordances}. We will denote this by $\calM^{\pos}_{\sqcup}(M\times
I)_g$. The topology of this moduli space will be our primary
focus. Precise definitions and relevant background material will be
given below. Note that as well as positive scalar curvature, the
curvature conditions we consider in this paper also include some
stronger `intermediate' notions of curvature. (See Definition
\ref{intermediate}).
%\begin{remark}
% \textcolor{blue}{There is a much fancier version of the moduli space
%  $\calM^{\pos}_{\sqcup}(M\times I)_g$ is given in
%    \cite{Ebert-Randal-Williams}, where the \emph{concordance
%   category $\calQ(M)$ of metrics with positive scalar curvature on
%    a manifold $M$} is introduced (see \cite[Theorem
 %     C]{Ebert-Randal-Williams}). In this paper, we use relatively
 %   straightforward geometrical tools to study this moduli space. We
 %   would like to emphasize that the constructions we use are quite
 %   general and work for any smooth manifold admitting a metric of
 %   positive scalar curvature. In particular, we do not use spin
 %   structures, Dirac operators, or techniques related to minimal
 %   hypersurfaces.}
%\end{remark}

In recent work (\cite{BW}), Watanabe and the first author studied
$\calM^{\pos}_{\sqcup}(D^n\times I)_g$, i.e. the case where $M$ is a
disc, and showed that its rational homotopy groups can be non-trivial
in an `unstable range'.

The main result in the current paper is to establish a far-reaching
stable analogue of this, which applies to compact even dimensional
$M$, both when $M$ is closed and when $M$ has non-empty boundary. We
will postpone the full version of the main theorem until we have
developed the relevant notions and notation, however the following is
a simplified version which deals with the positive scalar curvature
case:

\noindent{\bf Main Theorem (simplified form):} {\it Given $q \in \N$, provided the dimension of $M$ is even and sufficiently large depending on $q$, then $\pi_{4q}\calM^{\pos}_{\sqcup}(M\times I)_g \otimes \Q \neq 0.$}

We would like to emphasize that the constructions we use to study $\calM^{\pos}_{\sqcup}(M\times I)_g$ are quite general, and work for any smooth manifold admitting a metric of
positive scalar curvature. In particular, we do not use spin structures, Dirac operators, or techniques related to minimal hypersurfaces.
\begin{remark}
We note that an alternative, more homotopy theoretic, approach to concordances is given in
\cite{Ebert-Randal-Williams}. Here, the \emph{concordance
category $\calQ(M)$ of metrics with positive scalar curvature on
a manifold $M$} is introduced and studied. See in particular \cite[Theorem
C]{Ebert-Randal-Williams}.
\end{remark}

\medskip

In the remainder of the introduction we will develop the background concepts and context for the results in this paper. We begin, however, with some notation.
\smallskip

Let $M$ be a closed 
manifold. We denote by $\calR(M)$ the space of all Riemanian
metrics on $M$ (equipped with the smooth topology), and by
$\calR^{\pos}(M)\subset \calR(M)$ the subspace of metrics with positive
scalar curvature. 

If $M$ has non-empty boundary $\del M$, we fix a boundary metric $h$
on $\del M$ and assume that all metrics on $M$ agree up to infinite
order with a product metric $h+dr^2$ near the boundary. We denote by
$\calR_{\del}(M)_{h}$ and by $\calR^{\pos}_{\del}(M)_{h}$ the
corresponding spaces of metrics.

We also consider a cylinder $M\times I$, which is a manifold with
corners when $\partial M \neq \emptyset$. Then for a metric $g_0\in
\calR_{\del}(M\times \{0\})_{h}$, we denote by $\calR_{\sqcup}(M\times
I)_{g_0}$ the space of all metrics $\bar g$ agreeing up to infinite
order with $g_0+dt^2$ near $M\times \{0\}$, and with $h+dr^2+dt^2$
along $\del M\times I$. In other words $\bar g$ agrees up to infinite
order with a fixed metric at all points of
\begin{equation}\label{ast}
\bigl(M \times \{0\}\bigr) \cup \bigl(
\partial M \times [0,1]\bigr). %\eqno(\ast)
\end{equation}
In the case where $\partial M =\emptyset,$ we require the conditions
to hold at $M \times \{0\}$ only. Notice that in either case, the
fixed part of the metric is determined by its restriction to $M \times
\{0\},$ so in our notation it suffices to record the metric only on
the ``bottom'' of the cylinder, even when $M$ has non-empty boundary.

If we fix a further metric $g_1\in \calR_{\del}(M\times \{1\})_{h}$ in
the case where $\partial M \neq \emptyset$, or $g_1\in \calR(M\times
\{1\})$ if $\partial M =\emptyset$, we denote by $\calR_{\del}(M\times
I)_{g_0,g_1}$ the subspace of $\calR_{\sqcup}(M\times I)_{g_0}$
consisting of all metrics which also agree up to infinite order with
$g_1+dt^2$ at $M\times \{1\}$.

Two metrics $g_0,g_1$ belonging to $\calR^{\pos}_{\del}(M)_{h}$ or
$\calR^{\pos}(M)$ are declared to be positive scalar curvature
\emph{concordant} if there exists a metric $\bar g\in
\calR^{\pos}_{\del}(M\times I)_{g_0,g_1}$.  For a given metric $g$ in
$\calR^{\pos}_{\del}(M)_{h}$ or $\calR^{\pos}(M)$, we consider the
space
\begin{equation*}
\calC_{\sqcup}^{\pos}(M\times I)_{g}\subset \calR_{\sqcup}(M\times I)_{g}
\end{equation*}
of all positive scalar curvature concordances $\bar g$ which restrict
to $g$ on $M\times \{0\}$.

Informally speaking, ``three sides'' of the cylinder $M \times I$ are
fixed here when $\partial M \neq \emptyset$, and this gives rise to
the notation $\sqcup$ for denoting spaces of concordances.

Observe now that there is a natural diffeomorphism group which acts
(by pull-back) on the concordance space. If $M$ is closed, we need to
consider diffeomorphisms $\theta: M \times I \to M \times I$ for which
the $\infty$-jet agrees with that of the identity when restricted to
$M \times \{0\},$ and agrees with that of the product diffeomorphism
$\theta|_{M \times \{1\}} \times \text{id}_I$ on the ``upper''
boundary $M \times \{1\}$. Similarly, in the case of a manifold with
boundary we want the $\infty$-jet to agree with the identity when
restricted to the boundary submanifold (\ref{ast}), and agree with
$\theta|_{M \times \{1\}} \times \text{id}_I$ at $M \times \{1\}.$ In
either case, we will denote this group of diffeomorphisms by
$\Diff_{\sqcup}(M \times I)$. This is the group of pseudoisotopies of
the manifold $M$.

There is a stabilization map $ s: \Diff_{\sqcup}(M \times I)\to
\Diff_{\sqcup}(M \times I\times I) $ given by crossing with the
interval $I$ (and unbending corners).  The \emph{smooth pseudoisotopy
stable range} is the function
\begin{equation}\label{phi-n}
  \phi(n) := \max\left\{\ k \ | \ 
  \mbox{the map $s$ is $k$-connected for all $n$-manifolds $M$}\right\}.
\end{equation}
\begin{remark}
It is classical result that $\phi(n)\geq
\min\{\frac{n-7}{2},\frac{n-4}{3}\}$ (see Igusa \cite{I}). Recently
Randal-Williams proved that there is an upper bound $\phi(n)\leq n-2$
if $n$ is even and $n\geq 6$, see \cite{R-W}.
\end{remark}
Returning to concordances, we can now form the quotient space
\begin{equation*}
\calM^{\pos}_{\sqcup} (M \times I)_g:=\calC^{\pos}_\sqcup(M
\times I)_g/\Diff_\sqcup(M \times I).
\end{equation*}
This is the \emph{moduli space of positive scalar curvature
  concordances}, which will be our principal focus in this
paper. Similarly, we can form the moduli space
\begin{equation*}
\calM^{\pos}_\partial(M \times I)_{g_0,g_1}:=\calR_\partial(M \times I)_{g_0,g_1}/\Diff_\partial(M \times I),
\end{equation*}
where $\Diff_\partial(M \times I) \subset \Diff_\sqcup(M \times I)$ is
the subgroup consisting of diffeomorphisms which also agree with the
identity up to infinite order at $M \times \{1\}$. In the sequel, we
will often fix a metric $\tilde g$ on $\partial(M \times I)$ and
simply write $\calM^{\pos}_\partial(M \times I)_{\tilde{g}}.$
 
Before proceeding further, let us give precise definitions of some of
the concepts we will be working with in this paper. We begin with the
notions of curvature referred to in the opening paragraph.
\begin{definition}\label{intermediate}
Given an integer $1\leq k\le n$, we say that a
  Riemannian manifold $(M,g)$ has \emph{$k$-positive Ricci
curvature}, denoted $Sc_k>0$, if for all $p\in M$ the Ricci tensor
$\Ric$ on $T_p M$ is $k$-positive, i.e. the sum of any $k$ eigenvalues
is positive.
\end{definition}
Note that the condition $Sc_k>0$ interpolates between positive Ricci
curvature (when $k=1$) and positive scalar curvature (when $k=n$).
Clearly the condition $Sc_{k-1}>0$ implies $Sc_{k}>0$ if $k\leq n$.

When considering these intermediate curvatures, we will define
(moduli) spaces of metrics and concordances in the same way as for
positive scalar curvature. There is, however, one point of difference
that we need to be mindful of: if $(M,g)$ has $Sc_k>0$, then the
corresponding product metric on $M \times I$ only has $Sc_{k+1}>0.$

Let us fix $k$ with $2\leq k \leq n-1$. Then depending on whether
$\partial M \neq \emptyset$ or $\partial M=\emptyset,$ we denote by
$\calR_{\del}^{\pos_k}(M)_{h}\subset \calR_{\del}(M)_{h}$ respectively
$\calR^{\pos_k}(M)\subset \calR(M)$, the space of metrics $g$ on $M$
with $Sc_{k}(g)>0$.  Since our primary focus in this paper is
concordances, we will define the space of $k$-positive Ricci curvature
concordances,
\begin{equation*}
  \calC_{\sqcup}^{\pos_k}(M\times I)_{g}\subset \calR_{\sqcup}(M\times I)_{g}
\end{equation*}
to be the space of concordances $\bar g$ satisfying $Sc_k(\bar
g)>0$. Note that this implies that the metric $g$ on $M \times \{0\}$
must satisfy $Sc_{k-1}(g)>0$, hence the need for $k \ge 2.$

Again, the diffeomorphism group $\Diff_{\sqcup}(M\times I)$ acts on
the space $\calC_{\sqcup}^{\pos_k}(M\times I)_{g}$, and we obtain the
orbit space
\begin{equation*}
\calM^{\pos_k}_{\sqcup} (M \times I)_g:=\calC^{\pos_k}_\sqcup(M
\times I)_g/\Diff_\sqcup(M \times I).
\end{equation*}  
The space $\calM^{\pos_k}_{\sqcup} (M \times
I)_g$ is the moduli space of $Sc_k>0$ concordances `based' at $g$.

In the case where $\partial M=\emptyset$ we make the obvious modifications to the above and use the same notation. To simplify matters, we would like to adapt the following
\begin{Convention}
We choose either $k=n$ or $2\leq k \leq n-1$. Then we will use the
notation $Sc_*>0$ instead of $Sc_k$ and the notations
$\calR^{\pos_*}(M)$, $\calR_{\del}^{\pos_*}(M)_{h}$,
$\calC^{\pos_*}_{\sqcup}(M \times I)_g$, $\calM^{\pos_*}_{\sqcup}(M
\times I)_g$, etc, for the corresponding spaces of metrics and moduli
spaces of metrics.  Thus our notation will not distinguish between
positive scalar curvature and $k$-positive Ricci curvature ($k \neq
1$). Note that if we write $\bar{g} \in \calC_\sqcup^{\pos_*}(M \times
I)_g$ say, where we understand $\pos_*$ to mean $Sc_k>0,$ in writing
$g \in \calR^{\pos_*}(M)$ it is implicit that the latter $\pos_*$
should be taken to mean $Sc_{k-1}>0.$
\end{Convention}
As an aside, we remark that for a manifold with boundary $M$,
demanding that the metric at the boundary agrees up to infinite order
with a product $h+dr^2$ aligns with convention, and in particular with
\cite{BW}. However, this restriction is actually not necessary. We
wish to fix the form of the metric up to infinite order at the
boundary, but the precise form of this has no consequences in the
sequel. For example, if $M=D^n$, then we could choose the metric at
the boundary to agree with that of a strictly convex round `polar
cap'.
%\medskip

Our main result is the following.
\begin{maintheorem}\label{mainA}
  For any positive integer $q$, there exists $N(q)$, with $N(q) \to
  \infty$ as $q \to \infty$, such that if $M$ is a compact manifold
  of even dimension $n>N(q)$ which admits a
  metric with $Sc_*>0$, then
\begin{equation}\label{tm:main} 
  \pi_{4q}\calM^{\pos_*}_{\sqcup}(M \times I)_g \otimes \Q \neq 0
\end{equation} 
for any metric $g\in \calR^{\pos_*}_{\del}(M)_{h}$. 
\end{maintheorem}
\begin{remark}
We again emphasize that $M$ could be a closed manifold or a manifold
with boundary. It is well-known that the space $\calR^{\pos}(M)_{h}$
is weak-homotopy invariant under admissible surgeries. At the same
time, the topology of the positive scalar curvature moduli space
$\calM^{\pos}_{\del} (M)= \calR^{\pos}_{\del}(M)_{h}/\Diff_{\del}(M)$
may change dramatically after surgery.  Similarly, the concordance
space $\calR^{\pos}_\sqcup(M \times I)_g$ is weakly-homotopy invariant
under such surgeries, but we cannot expect a similar result to be true
for the moduli space $\calM^{\pos}_{\sqcup}(M \times I)_g$. However, the
Main Theorem holds for an arbitrary manifold $M$ of even dimension.
\end{remark}  
\begin{remark}
In \cite{BW}, the first author and Watanabe studied the moduli space
$\calM^{\pos}_{\del} (D^n)_{h_{st}}$, where $h_{st}$ is the standard
round metric on $S^{n-1}$.  The main result in this direction is that
for $n \ge 4$ and $k \ge 2$, surgery on trivalent graphs produces
non-trivial elements in $\pi_{k(n-3)}\BDiff_\partial (D^n) \otimes
\Q$, and these elements lift to non-zero elements in
$\pi_{k(n-3)}\calM^{\pos}_{\del}(D^n)_{h_{st}}\otimes \Q$. In turn,
the latter elements give rise to non-zero elements in
$\pi_{k(n-3)}\calM^{\pos}_\sqcup(D^n \times I)_g \otimes \Q$
demonstrating the topological non-triviality of the moduli space of
positive scalar curvature concordances of the disc. It is important to
emphasize that these results hold in the ``unstable range''. It is
natural to expect that similar facts could be established for an
arbitrary manifold $M$. However, those elements seem vanish under a
natural map
\begin{equation*}
  D^{2\ell} \rh D^{2\ell}\cup_{S^{2\ell-1}\times \{0\}} [(S^{2\ell-1}\times I)\#_m
    (S^{\ell}\times S^{\ell})]
\end{equation*}  
which takes connected sums with $S^{\ell}\times S^{\ell}$ for large enough $m$.
\end{remark}
\begin{remark}
  Clearly our notion of concordance does not make sense for
  Ricci-positive metrics, since the product structure near the
  boundary produces a zero eigenvector for the Ricci tensor: this is
  why we exclude the case $k=1$.  We also note that Burdick has
  introduced a viable notion of positive Ricci curvature concordance
  in \cite{B}, however at the time of writing, this concept has yet to
  be explored in a significant way.
\end{remark}
The proof of Main Theorem ultimately rests on certain models for
classifying spaces, and constructions made with these models. These
constructions involve Hatcher bundles, from where the non-triviality
in rational homotopy arises. More specifically, we use the existence
of Hatcher bundles with fibrewise positive Ricci curvature metrics
established in \cite{BWW}. However see also \cite{BHSW}, where the
role played by Hatcher bundles is closer to that in the current paper.

This paper is laid out as follows. In Section \ref{models} we
construct the models for classifying spaces discussed above. The most
challenging technical aspect to this involves showing that certain
spaces of embeddings are contractible, (see Lemma \ref{contractible}),
and the arguments here might be of independent interest. In Section
\ref{lifting} we review the implications for homotopy groups which
derive from the existence of fibrewise positively curved metrics on
disc bundles. Section \ref{Hatcher} discusses Hatcher bundles and
establishes the main properties of these objects that we will
need. Finally in Section \ref{concordance} we prove Main Theorem 
by a combination of further constructions and diagram-chasing
arguments.

The foundations of this paper were laid when the second author was
visiting the University of Oregon, and he would like to thank the
University of Oregon for their hospitality and support. The first author
was partially supported by Simons collaboration grant 708183.

%%%%%%%%%%

\section{Models for classifying spaces}\label{models}
In this section we will study two classifying spaces:
$\text{BDiff}_\partial(M \times I)$ and $\text{BDiff}_\sqcup(M \times
I)$. Let us first consider $\text{BDiff}_\sqcup(M \times I).$ We will
present two models for this space.

Consider the space of all Riemannian metrics on $M \times I$, assuming
an infinitessimal product structure in the $I$-direction at each of
the boundary components $M \times \{0\}$ and $M \times \{1\}$, and in
the case where $M$ has boundary, also in the inward normal directions
at all points of $\partial M \times I$. Assume the metric agrees up to
infinite order with some product $g+dt^2$ at the
lower boundary $M \times \{0\}$, and if $\partial M \neq \emptyset$,
with a product $h+dr^2+dt^2$ near $\partial M \times
I$, where $h=g|_{\partial M}$ and
$r$ is the inward normal parameter in each `slice' $M \times
\{t\}$. We denote this space by
  $\calR_\sqcup(M \times I)_g$, as above, and observing that it is both
contractible and acted upon freely by $\text{Diff}_\sqcup(M \times
I)$, we arrive at our first model for the classifying
space:
\begin{equation*}
\text{BDiff}_\sqcup(M \times I)=\calR_\sqcup(M \times
I)_g/\text{Diff}_\sqcup(M \times I).
\end{equation*}
Our second model is a space of embeddings of $M \times I$ into
$\R^\infty$.  Fix an embedding $e:M \rh \R^d$ for some $d$, extend to
an embedding $\tilde{e}:M \times I \rh \R^{d+1}$ given by
$\tilde{e}(x,t)=(e(x),t)$, and define the embedding $\bar{e}:M \times
I \rh \R^{\infty}$ as the composition
\begin{equation*}
M \times I
\lhook\joinrel\xrightarrow{\ \tilde{e}\ } \R^{d+1}=\R^{d+1} \times \{0\}
\hookrightarrow
\R^{\infty}.
\end{equation*}
Consider the space of embeddings $\text{Emb}_\sqcup(M \times
I,\R^{\infty})$ consisting of those embeddings $\phi:M \times I \to
\R^{\infty}$ for which the $\infty$-jet at the lower boundary $M=M
\times \{0\}$ agrees with that of $\bar{e}$, and if $\partial M \neq
\emptyset$ we also demand that the embeddings agree up to infinite
order with $\bar{e}$ along the `sides' $\partial M \times I$. We will
also assume that the upper boundary agrees with that of an embedding
taking the form $\psi(x,t)=\tau_{(t-1)v} \circ \phi(x,1),$ where
$\tau_w:\R^\infty \to \R^\infty$ is translation through the vector $w
\in \R^\infty$, and $v$ is any vector transverse to all tangent spaces
of $\phi(x,1).$ In other words, we require the `top' of the embedding
to have an infinitessimal product-like structure.

Similarly, let $\text{Emb}_\partial(M \times I,\R^{\infty})\subset
\text{Emb}_\sqcup(M \times I,\R^{\infty})$ be the subset for which the
$\infty$-jet at the upper boundary $M \times \{1\}$ also agrees with
that of $\bar{e}$.
\begin{lemma}\label{contractible}
The embedding spaces $\Emb_\sqcup(M \times I,\R^{\infty})$ and
$\Emb_\partial(M \times I,\R^{\infty})$ are contractible.
\end{lemma}
%As the proof of this Lemma is quite substantial,
We will postpone the proof of this Lemma until the
 end of this section.

Next, observe that the group of pseudoisotopies
$\text{Diff}_\sqcup(M\times I)$ acts freely on the space
\begin{equation*}
  \text{Emb}_\sqcup(M \times I,\R^{\infty}),
\end{equation*}
and thus we obtain a model for the classifying space
$\text{BDiff}_\sqcup(M\times I)$: $$\text{BDiff}_\sqcup(M\times
I)=\text{Emb}_\sqcup(M \times
I,\R^{\infty})/\text{Diff}_\sqcup(M\times I).$$ We conclude that any
element of $\text{BDiff}_\sqcup(M\times I)$ can be identified with the
{\it image} of an embedding belonging to $\text{Emb}_\sqcup(M \times
I,\R^{\infty})$.

In a similar fashion, observe that $\text{Emb}_\partial(M \times
I,\R^{d+\infty})$ is acted upon by $\text{Diff}_\partial(M \times I),$
the group of diffeomorphisms which agree up to infinite order with the
identity at all boundary points, and that this action is free. Thus we
arrive at the model
\begin{equation*}
  \text{BDiff}_\partial(M \times
I)=\text{Emb}_\partial(M \times
I,\R^{\infty})/\text{Diff}_\partial(M\times I).
\end{equation*}
For the metric analogue of this last classifying space, we consider
the space $\calR_\partial(M \times I)_g$ of all
Riemannian metrics $\bar g$ on $M \times I$ which
agree up to infinite order with a product metric $g+dt^2$ on both $M
\times \{0\}$ and $M \times \{1\}$. If $\partial M \neq \emptyset,$ we
also demand that the metric agrees up to infinite order with
$h+dr^2+dt^2$ along $\partial M \times I$, with
  $h=g|_{\del M}$. Again, this space is contractible and is acted
upon freely by $\text{Diff}_\partial(M \times I)$. Hence we
obtain
\begin{equation*}
  \text{BDiff}_\partial(M \times I)=\calR_\partial(M\times
I)_g/\text{Diff}_\partial(M \times I).
\end{equation*}
In the sequel we will also need to consider the classifying space
$\text{BDiff}_\partial (D^n)$. We can realize this space via both
embedding and metric models, in a manner analogous to
$\text{BDiff}_\partial(M \times I)$.  For the embedding model, we
consider the embedding $\iota: D^n \rh \R^{n+1}\times\{0\}\subset
\R^\infty$ as the inclusion of a standard unit round hemisphere, and
then let $\text{EDiff}_\partial(D^n)$ be the space of all embeddings
of $D^n$ into $\R^\infty$ which agree up to infinite order with
$\iota$ at the boundary. 
% In general, we fix an embedding $\iota_0 :\partial M \rh \R^{n+\ell}$ for some $\ell$ for a manifold $M$ with non-empty boundary.
  
Similarly, in the metric model, we assume that all metrics in
$\text{EDiff}_\partial(D^n)$ agree up to infinite order with the unit
round hemisphere metric at the boundary.\footnote{
  \ We will omit the boundary metric in the notation for this model.} We
have a free action of $\text{Diff}_\partial(D^n)$ on either of these
contractible spaces, with the classifying space being the resulting
quotient. Note that in either model, there is nothing special about
the choice of a hemisphere here, other than for the convenience of
language. We could, for example, consider round `polar caps' of any
radius. We will, in fact, need to do this in Section \ref{m+}.

In later constructions, it will be convenient to move freely between
embedding and metric models for the classifying spaces we
consider. This is facilitated by a natural map from the embedding to
the metric model, which we will now describe. Given a manifold $N$,
(which will represent either $M \times I$ or $D^n$), let
$\BDiff^{\e}_\partial (N)$ denote the embedding model, and
$\BDiff^{\m}_\partial (N)$ the metric model.

An element $y \in \BDiff^{\e}_\partial (N)$ corresponds to an embedded
submanifold in $\R^\infty$ with certain boundary conditions, which is
the image of the any of the embeddings represented by $y$. We can
restrict the Euclidean inner product on $\R^\infty$ to obtain an
induced Riemannian metric on the embedded submanifold. This metric can
be pulled back by any of the embeddings determined by $y$ to obtain an
isometry class of metrics on a standard copy of $N$ satisfying
appropriate boundary conditions. This latter class represents an
element in $\BDiff^{\m}_\partial (N),$ and so we obtain a \emph{metric
restriction map} $\xi:\BDiff^{\e}_\partial (N) \to
\text{BDiff}^{\m}_\partial (N).$
\begin{lemma}\label{embedding_to_metric}
The map $\xi:\BDiff^e_\partial (N) \to \BDiff^{\m}_\partial (N)$ is a
homotopy equivalence.
\end{lemma}
Before proving this, let us recall that there is a bijection between
principal $G$-bundles over a space $X$ and the set of homotopy classes
$[X,BG]$, given by pulling back the universal bundle $EG \to BG.$ In
particular, if $X$ is an alternative model for $BG$, then the
universal $G$ bundle over $X$ can be obtained by pulling back the
universal bundle over $BG$ via some map $f$. Similarly, the universal
bundle over $BG$ can be obtained by pulling back the universal bundle
over $X$ by some map $g$. Clearly, pulling back by the compositions $f
\circ g$ and $g \circ f$ has the same effect as pulling back by the
respective identity mappings, hence these compositions are homotopic
to the identity. Thus $f$ and $g$ are homotopy equivalences between
the two models. To prove the above lemma, it therefore suffces to show
that the pull-back of $\EDiff_\partial^{\m}(N) \to
\BDiff_\partial^{\m}(N)$ along $\xi$ is the bundle
$\EDiff_\partial^{\e}(N) \to \BDiff_\partial^{\e}(N).$

Note that a fibre of $\EDiff_\partial^{\e}(N) \to
\BDiff_\partial^{\e}(N)$ consists of the collection of embeddings with
a given common image, and a fibre of $\EDiff_\partial^{\m}(N) \to
\BDiff_\partial^{\m}(N)$ is the set of metrics which make up a
given isometry class.
\begin{proof}[Proof of Lemma \ref{embedding_to_metric}]
As noted above, it suffices to show that the pull-back of
the fibre bundle $\EDiff_\partial^{\m}(N) \to
\BDiff_\partial^{\m}(N)$ by $\xi$ coincides with the
  bundle $\EDiff_\partial^{\e}(N) \to \text{BDiff}_\partial^e(N).$

To this end, we simply observe that the metric restriction map $\xi$
defined above is covered by a map
$\bar{\xi}:\EDiff^{\e}_\partial(N) \to \EDiff^{\m}_\partial(N)$,
which maps a given embedding $\phi$ to the metric on the standard copy
of $N$ obtained by pulling back the induced metric on $\phi(N)$ via
$\phi$. Clearly, $\bar{\xi}$ is a smooth
$\Diff_\partial(N)$-equivariant bijection between the fibre over
any $y \in \BDiff^{\e}_\partial (N)$ and the fibre over $\xi(y) \in
\BDiff_\partial^{\m}(N).$ Thus $\bar{\xi}$ induces an equivariant
diffeomorphism between the fibre over $y$ in
$\EDiff^{\e}_\partial(N)$ and the fibre over $y$ in
$\xi^*(\EDiff^{\m}_\partial(N)).$ Thus the bundles
$\EDiff^{\e}_\partial(N)$ and $\xi^*(\EDiff^{\m}_\partial(N))$
are equivalent, as claimed.
\end{proof}
We conclude this section with the proof of Lemma
\ref{contractible}. To preface this, note that the space of embeddings
of a given manifold in $\R^\infty$ without any boundary conditions is
known to be contractible. See for example \cite[Appendix A]{H}.  The
argument is essentially this: compose the embeddings with the obvious
linear homotopy between the identity map of $\R^\infty$ and the map
which sends the $k^{th}$ basis vector to the $(2k-1)^{st}$. This
pushes all embeddings into odd dimensions. One then fixes an embedding
into even dimensions, and slides the odd dimensional embeddings
to the fixed even dimensional embedding by a second obvious linear
homotopy. In our case, however, we need to retain boundary conditions
throughout such a construction, and this leads to technical
difficulties which the proof below shows how to overcome.
\begin{proof}[Proof of Lemma \ref{contractible}.]
The aim here is to construct a homotopy of the space of embeddings
\begin{equation*}
\Emb_\bullet(M \times I,\R^{\infty}),
\end{equation*}  
(where $\bullet=\sqcup,\partial$), through $\text{Emb}_\bullet(M
\times I,\R^{\infty})$, to a fixed embedding, thus demonstrating the
desired contractibility.

Let $e:M \to \R^d$, $\tilde{e}:M \times I \to \R^{d+1}$ and $\bar{e}\in \Emb_\bullet(M \times I,\R^{\infty})$ be the 
embeddings defined at the start of this section. Without loss of generality, suppose that the target
space $\R^{d+1}$ of the embedding $\tilde{e}$ is embedded into
$\R^\infty$ as the first $d+1$ odd dimensions. In other words, if we
write $e=(e_1,...,e_d),$ so that $\tilde{e}(x)=(e_1(x),...e_d(x),t),$
then $\bar{e}(x)=(e_1(x),0,e_2(x),0,...,e_d(x),0,t,0...)$. 
Let us be clear on the role that the embedding $\bar{e}$ will play in this proof. Firstly of course, it
provides us with the boundary conditions common to all embeddings in
$\text{Emb}_\bullet(M\times I,\R^{\infty})$. Secondly, in the
paragraph below, we will show how to deform an arbitrary embedding
$\phi \in \text{Emb}_\bullet(M\times I,\R^{\infty})$ into an embedding
into odd dimensions only, and this homotopy will use the dimension $d$
which arises from the initial embedding $e$. Finally, we will show
that $\text{Emb}_\bullet(M\times I,\R^{\infty})$ can be continuously
deformed into a one-point space whose single element is an embedding
obtained via a slight modification of $\bar{e}$.

For the first part of the construction, we will exhibit a homotopy which will deform each embedding into odd dimensions, whilst preserving the boundary conditions. We will achieve this via a linear isotopy of $\R^{\infty}$. As a preliminary to defining this isotopy, consider the following injective map $f:\N \to \N$ given by
$$f(n)=
\begin{cases}
n \text{ for } n \text{ odd, } 1 \le n \le 2d+1; \\
2d+n+1 \text{ for } n \text{ even, } 2 \le n \le 2d; \\
2n-1 \text{ for } n \ge 2d+2. \\
\end{cases}
$$
Now define the isotopy $\theta_s:\R^{\infty} \to \R^{\infty}$, $s \in [0,1],$ as follows: for any $v \in \R^\infty$ with $i^{th}$ entry $v_i$, we map $$v_i \mapsto (1-s)v_i+sv_{f^{-1}(i)},$$ where we interpret $v_{f^{-1}(i)}$ as 0 if $i \not\in \text{im}f.$ We claim that this isotopy is an injection for each $s$. To see this, suppose that $v,w \in \R^\infty$ differ in the $i^{th}$ entry. If $i$ is even, or $i$ is odd and lies either in the in the range $1 \le n \le 2d+1$ (which is fixed by $f$) or in the range $2d+3 \le i \le 4d-1$ (which is not in the image of $f$), then by definition of $\theta_s$ we see immediately that the $i^{th}$ entries differ for all $s \in [0,1]$. This leaves the possibility that $i$ is odd and $i \ge 4d+1$. In this case, $(1-s)v_i+sv_{f^{-1}(i)} =(1-s)w_i+sw_{f^{-1}(i)}$ for some $s\in [0,1]$ only if that $v_{f^{-1}(i)} \neq w_{f^{-1}(i)}.$ Repeating the above arguments for the entries in position $f^{-1}(i)$, and iterating as necessary, we must eventually come to situation where the entries are different for all $s$, from which the injectivity of $\theta_s$ follows.

Given that $\theta_s$ is injective, the composition of any embedding into $\R^{\infty}$ with $\theta_s$ is an embedding for every $s$. Moreover $\theta_1$ is clearly a map into odd dimensions, as required.

Turning our attention to boundary conditions, note that for all $\phi \in\text{Emb}_\bullet(M \times I,\R^{\infty}),$ $\theta_s \circ \phi$ agrees with $\phi$ in all odd dimensions $n$ with $1 \le n \le 2d+1$. In the remaining dimensions, each entry in the image of $\theta_s \circ \phi$ is a linear combination of those for $\phi$ from dimensions excluding the odd dimensions $1 \le n \le 2d+1$. Given that $\phi$ agrees up to infinite order on the boundary (or part of the boundary) with $\bar{e}$, it is automatic that $\theta_s \circ \phi$ satisfies the same boundary conditions as $\phi$. Similarly, if there are no fixed boundary conditions on $M \times \{1\}$, then it  is a triviality that composition with $\theta_s$ preserves the infintessimal product-like structure at the upper boundary.

Thus for first stage in the homotopy of our space of embeddings, we set $$H_1(\phi(x,t),s)=\theta_s \circ \phi(x,t).$$

Having achieved a homotopy of the space $\text{Emb}_\bullet(M \times I,\R^{\infty})$ to a subspace of embeddings lying in odd dimensions, our final task is to show how to continuously deform this subspace onto a one-point space, i.e. how to deform every embedding (in a uniform way) to a fixed embedding in $\text{Emb}_\bullet(M \times I,\R^{\infty})$. We begin by defining this fixed embedding, which we will label $\phi_0$.

This fixed embedding will be a modification of $\bar{e}$. The aim of this modification is to `nudge' $\bar{e}$ into even dimensions away from the boundary. To achieve this deformation, we introduce the following function. Let $\psi:\R \to \R$ be a smooth function, strictly increasing on $[0,1]$, which satisfies the following conditions: $\psi(t)=0$ for $t \le 0$, $\psi(t)=1$ for $t \ge 1$, and $\psi^{(n)}(0)=\psi^{(n)}(1)=0$ for all $n \in \N$.

The definition of $\phi_0$ will depend on the type of boundary conditions under consideration. There are three cases: firstly, the case where $\partial M =\emptyset,$ so the boundary conditions apply to $M \times \{0\}$ and possibly $M \times \{1\}$ only. The second is where $\partial M\neq \emptyset$ and the boundary conditions apply to {\it all} boundary components, i.e. to $(M \times \{0\}) \cup (\partial M \times I) \cup (M \times \{1\}).$ The final case is where $\partial M \neq \emptyset,$ and the boundary conditions apply to all components {\it except} $M \times \{1\}$.

Let us initially assume that $\partial M =\emptyset$, so we have boundary conditions on $M \times \{0\}$ and perhaps also $M \times \{1\}$. Fix a very small $\epsilon>0$, and a background metric $g$ on $M$. First assume that $\partial M =\emptyset.$ With respect to the metric $g+dt^2$ on $M \times I$, consider the distance function $\delta:M \times I \to [0,\infty)$ defined by setting $\delta(x,t)$ to be the minimum distance with respect to $g+dt^2$ of the point $(x,t)$ from $M \times \{0\}$, or from $(M \times \{0\}) \cup (M \times \{1\}),$ depending on whether the upper boundary component is being considered. Notice that $\delta$ is independent of $x$, the function $f_1(x,t):=\psi(\delta(x,t)/\epsilon)$ is smooth on $M \times I$, and that $f_1$ vanishes up to infinite order on the boundary components which are subject to the boundary conditions.

Let us define $\phi_0$ as follows: $$\phi_0:=(e_1,f_1e_1,...,e_d,f_1e_d,t,tf_1,0,f_1,0,0,0...).$$ Notice that when $\phi_0$ is composed with the projection $\pi_{odd}:\R^\infty \to \R^\infty$ onto odd dimensions, we recover the map $\bar{e}$, which shows that $\phi_0$ is an embedding. Moreover, by construction $\phi_0$ agrees with $\bar{e}$ up to infinite order on the relevant boundary components. 

We also claim that composing $\phi_0$ with the projection map onto {\it even} dimensions yields an embedding provided we restrict to the complement of the `fixed' boundary components. The key to seeing this is to examine the entries in positions $2d+2$ and $2d+4$. 

First note that the $t$-derivative of $tf_1$ (in entry $2d+2$) is equal to 1 at all points at distance at least $\epsilon$ from the relevant boundary components, since $f_1$ is constant with value 1 at such points. At distance less than $\epsilon$, the entry in position $2d+4$ has non-zero derivative (away from fixed boundaries. Thus $\pi_{even}\circ \phi_0$ has full rank (i.e. rank 1) when restricted to the $t$-direction in $M \times I$.

Next, notice that taken together, this pair of entries $(tf_1,f_1)$ uniquely determine $t$, since $f_1>0$ away from the fixed boundaries. Thus we see that $\pi_{even}\circ \phi_0$ is an embedding if it is an embedding when restricted to each `slice' $M \times \{t\}$. But this is clear, since for each $t$ we have agreement with the embedding $e$ in the first $d$ coordinates up to scaling by $f_1$, which is constant on each slice. 

In the case where $M \times \{1\}$ is not subject to fixed boundary conditions, we also need to check the infinitessimal product structure at this boundary component. For $t \approx 1$ we have $f_1(x,t)=1$, and so $\phi_0(x,1)-\phi_0(x,t)=(0,...,0,1-t,1-t,0,0,...)$, showing the desired structure.

Having constructed our fixed embedding $\phi_0$, the last stage of the contractibility argument is to construct a second homotopy to show that our space of embeddings in odd dimensions has the same homotopy type as the one-point set $\{\phi_0\}.$ We do this by setting $$H_2(\phi(x,t),s)=(1-s)\phi(x,t)+s\phi_0(x,t),$$ where it is assumed that $\phi(x,t)$ lies in the image of the homotopy $H_1$. We claim that for each $s \in [0,1],$ $H_2(\phi(x,t),s)$ is an embedding $M \times I \to \R^\infty.$ Since the claim for $s=0,1$ is already clear, we focus on the case $s \in (0,1)$. As noted above, the composition of $\phi_0$ with projection onto the even dimensions results in an embedding of the complement of the fixed boundary component(s). Since $\phi(x,t)$ by assumption has only zero entries in the even dimensions, away from the fixed boundary component(s) of $M \times I$, the projection of $H_2(\phi(x,t),s)$ onto even dimensions must agree with $s\pi_{even}\circ\phi_0(x,t)$, and therefore is an embedding for any $s \in (0,1)$. Next, observe that $H_2(\phi(x,t),s)$ is $C^\infty$-arbitrarily close to the embedding $\phi$ in a suitably small neighbourhood of the fixed boundary components. Since embeddings are stable under small perturbations, $H_2(\phi(x,t),s)$ must also be an embedding when restricted to such a neighbourhood. To complete the proof of the claim, it now remains to show that $H_2(\phi(x,t),s)$ is globally injective. Since $H_2(\phi(x,t),s)$ is injective both on the fixed boundary component(s) and on their complement, the only thing we need to check is that the images of both these regions are disjoint. However this is immediate by examination of the image coordinate in dimension $2d+2$. This is zero on the fixed boundary components, but non-zero on their complement. 

In the case where the boundary conditions apply to $M \times \{0\}$
only, observe that the infinitessimal product-like structure at $M
\times \{1\}$ is preserved by $H_2$.

Overall therefore, the concatenation of the homotopies $H_1$ and $H_2$
continuously deform the space $\text{Emb}_\bullet(M \times
I,\R^{\infty})$ when $\partial M=\emptyset$ onto a one-point subset,
showing contractibility.
%\medskip

It now remains to consider the situation where $\partial M \neq
\emptyset.$ As mentioned previously, we will split this into the two
possible cases, defining $\phi_0$ in each case, and pointing out how
the rest of the contractibility argument differs from that given
above.

We will again need the function $\delta$, which of course measures
distance (within $M\times I$) from $M \times \{0\}$ or $(M
\times \{0\})\cup (M \times \{1\}),$ but since $\partial M$ is now
non-empty, we will need an analogous function measuring distance from
$\partial M \times I$ with respect to the background metric
$g+dt^2$. Let $\bar{\delta}: M \times I \to [0,\infty)$ be this
  function, and set $f_2(x,t):=\psi(\bar{\delta}(x,t)/\epsilon).$
  Notice that this function is independent of $t$, vanishes on
  $\partial M \times I$, and assuming $\epsilon$ is suitably small, is
  smooth on $M \times I$.

In the case $\partial M \neq \emptyset$ we define $\phi_0$ as
follows:
\begin{equation*}
  \phi_0(x,t):=(e_1,f_1f_2e_1,...,e_d,f_1f_2e_d,t,\psi(t)f_1f_2,0,f_1f_2,0,0,0...).
\end{equation*}
Since $\pi_{odd}\circ \phi_0$ agrees with $\bar{e}$, we see
immediately that this function is an embedding. Moreover, since all of
the even dimensional entries vanish to infinite order on the boundary,
we see that all the boundary conditions are satisfied.

We claim that projection onto even dimensions also yields an embedding
when restricted to the complement of the boundary components. To this
end, notice that the $t$-derivatives of $\pi_{even}\circ \phi_0$ are
everywhere non-zero, by the same argument as in the case $\partial
M=\emptyset.$ The pair $(\psi(t)f_1f_2,f_1f_2)$ of entries in
positions $2d+2$ and $2d+4$ uniquely determine the $t$-coordinate
since $\psi(t)$ is injective, and from these facts we conclude that
$\pi_{even}\circ \phi_0$ is an embedding if it is an embedding when
restricted to each slice $M \times \{t\}$ for $t \in (0,1).$ On each
of these slices $f_1$ is constant, so we can effectively ignore it,
however $f_2$ is not constant. At distances of at least $\epsilon$
from the boundary, the restriction of $\pi_{even}\circ \phi_0$ to $M
\times \{t\}$ has full rank, since $f_2 \equiv 1$ at these distances,
and hence the map agrees with $e$ up to a constant factor. At
distances less than $\epsilon$ we have that $f_2$ is variable, however
observe that $f_2$ {\it is} constant on each equidistant hypersurface
to the boundary $\partial M \times \{t\} \subset M \times
\{t\}$. Since the restriction of $\pi_{even}\circ \phi_0$ to any such
hypersurface agrees with $e$ up to a constant factor, we conclude that
this is an embedding. In the normal direction to the boundary at
distances less than $\epsilon$ (within $M \times \{t\}$) we see that
the entries in both dimensions $2d+2$ and $2d+4$ have non-zero
derivative, and hence the restriction of $\pi_{even}\circ \phi_0$ to
$M \times \{t\}$ is a local diffeomorphism. It remains to argue that
this restricted map is injective: at distances at least $\epsilon$
this follows from the injectivity of $e$, at distances less than
$\epsilon$ from the injectivity of the entries in dimensions $2d+2$
and $2d+4$ at these distances, and overall from the fact that these
latter entries take a different value in the two regions.

The rest of the contractibility argument now follows through as in the
$\partial M =\emptyset$ case.

Finally, in the case where the fixed boundary conditions do not apply
to $M \times \{1\},$ we still have to check the existence of an
infinitessimal product-like structure at this boundary component. For
$t\approx 1$ we have
$\phi_0(x,1)-\phi_0(x,t)=(0,...,0,1-t,f_2(1-\psi(t)),0,0,...).$ The
$t$-derivative of the entry $f_2(1-\psi(t))$ at $t=1$ is
$-f_2\psi'(1)=0$. Similarly for the higher $t$-derivatives. Thus up to
infinite order, $\phi_0(x,t)$ agrees with $\phi(x,1)+(1-t)v_{2d+1}$ at
$t=1$, where $v_{2d+1}$ denotes the standard basis vector in dimension
$2d+1$. This establishes the desired infinitessimal product-like
structure.
\end{proof}

%\bigskip\bigskip

%%%%%%%%%%

\section{Lifting elements of $\pi_q\BDiff_\partial(D^n)\otimes \Q$ to $
\pi_q\calM^{\pos_*}(D^n)_{h_0} \otimes \Q$.}\label{lifting}

Consider an element $\kappa \in \pi_q \BDiff^{\e}_\partial (D^n)$
represented by a based map $\mathsf{k}: S^q \to \BDiff^{\e}_\partial
(D^n)$.  We can view any such map as an $S^q$-indexed family of disc
submanifolds of $\R^{\infty}$ with the prescribed boundary conditions,
i.e. a disc bundle over $S^q$ which is trivial when restricted to the
boundary.
\begin{remark}
We can also look at this bundle in another way: the map $\mathsf{k}$
determines a $\Diff_\partial (D^n)$-bundle over $S^q$, by pulling back
the universal $\Diff^{\e}_\partial (D^n)$-bundle over $\BDiff_\partial
(D^n)$. Via the associated bundle construction, we then re-obtain our
$D^n$-bundle over $S^q$ in the obvious way. To see that the latter
disc bundle agrees with the original, notice that by construction,
each fibre of the $\Diff_\partial (D^n)$-bundle consists of the set of
parametrizations of the corresponding fibre of the disc bundle. Thus
the $\Diff_\partial (D^n)$-bundle is precisely the associated
principal bundle to the disc bundle.
\end{remark}
On such a disc bundle we obtain a fibrewise Riemannian metric by
  restricting the metric on $\R^\infty$ to each embedded fibre
  disc. In other words we obtain a map $\bar{\mathsf{k}} : S^q \to
  \BDiff^{\m}_\partial (D^n)$, which is simply the composition
  $\xi\circ \mathsf{k}$. On the other hand, if we {\it impose} a
  smoothly varying fibrewise metric on this bundle, then the resulting
  collection of metrics determines a map $\tilde{\mathsf{k}}: S^q \to
  \BDiff^{\m}_\partial (D^n)$.  But notice that the two families of
  fibrewise metrics, and hence the two maps $\tilde{\mathsf{k}}$ and
  $\bar{\mathsf{k}}$, are homotopic via a linear homotopy. Thus both
maps represent the same element in $\pi_q\BDiff^{\m}_\partial (D^n).$
We conclude:
\begin{observation}\label{e_vs_m}
To each element $\kappa \in \pi_q\BDiff_\partial^{\e}(D^n),$ the
element $\xi_\ast(\kappa) \in \pi_q\BDiff_\partial^{\m}(D^n)$
represents all possible fibrewise metrics on a bundle of embedded
discs (as above) representing $\kappa$.
Moreover, the same
conclusion holds if $\kappa \in
\pi_q\BDiff_\partial^{\e}(D^n)\otimes \Q$ and $\xi_*$ denotes the
induced isomorphism of rational homotopy groups.
\end{observation}
In the Riemannian metric model for the classifying space
$\BDiff_\partial (D^n)$, we have a quotient map
$\calR_{\del}(D^n)_{h_0} \to \BDiff^{\m}_\partial (D^n)$, where $h_0$
denotes the unit radius round hemisphere metric on $D^n$ and
$\calR_{\del}(D^n)_{h_0}$ is the space of all Riemannian metrics on
$D^n$ which agree with $h_0$ up to infinite order at the boundary.
Since pulling back the metric preserves the curvature condition, we
therefore obtain a well-defined restriction map
$\calR_{\del}^{\pos_*}(D^n)_{h_0} \to \BDiff^{\m}_\partial (D^n).$ In
a similar fashion we also have an inclusion map of moduli spaces
$\calM_\partial^{\pos_*}(D^n)_{h_0} \to
\calM_\partial(D^n)_{h_0}=\BDiff^{\m}_\partial (D^n)$.

If there exists a fibrewise metric with positive curvature of the
appropriate type, we obtain a lift of the map $\xi\circ\mathsf{k}: S^q \to
\BDiff^{\m}_\partial (D^n)$ as follows:
\begin{center}
\begin{tikzcd}
  & & \calM_\partial^{\pos_*}(D^n)_{h_0} \arrow[d]
  \\
  S^q \arrow[rru,dashed]\arrow{r}[swap]{\mathsf{k}}&\BDiff^{\e}_\partial (D^n)
  \arrow{r}[swap]{\xi} & \BDiff^{\m}_\partial (D^n).
\end{tikzcd}
\end{center}
By Observation \ref{e_vs_m} it makes sense to identify
$\pi_q\BDiff_\partial^{\e}(D^n)$ and
$\pi_q\BDiff_\partial^{\m}(D^n)$ via $\xi_\ast$, and we will simply
write $\pi_q\BDiff_\partial(D^n)$ from now on.
In summary, we obtain:
\begin{lemma}\label{moduli_lifts}
The existence of a fibrewise ``positively curved'' metric (in whatever
sense is under consideration) on a $D^n$-bundle arising from an
element $\kappa \in \pi_q \BDiff_\partial (D^n)$, (respectively
  an element $\kappa \in \pi_q \BDiff_\partial (D^n)\otimes
  \Q$), determines an element $\bar{\kappa} \in \pi_q
\calM_\partial^{\pos_*}(D^n)_{h_0}$, (respectively an
    element $\bar{\kappa} \in \pi_q
  \calM_\partial^{\pos_*}(D^n)_{h_0}\otimes \Q)$, which coincides
with $\xi_\ast(\kappa)$ under the map induced by the inclusion
\begin{equation*}
  \calM_\partial^{\pos_*}(D^n)_{h_0} \rh \calM_\partial(D^n)_{h_0}=\BDiff_\partial (D^n).
\end{equation*}  
\end{lemma}
%%
%
%, (respectively
%$\kappa \in \pi_q \BDiff_\partial (D^n)\otimes \Q$), determines
%an element $\bar{\phi} \in \pi_q \calM_\partial^{\pos_*}(D^n)_{h_0}$,
%(respectively $\bar{\phi} \in \pi_q {\mathcal M}^+(D^n)_{h_0}\otimes
%\Q$), with image $\xi_\ast(\kappa)$ under the map induced by the
%inclusion ${\mathcal M}_\partial^+(D^n)_{h_0} \to {\mathcal
%  M}_\partial(D^n)_{h_0}=\text{BDiff}_\partial (D^n)$.
%
%The existence of a fibrewise `positively curved' metric (in whatever
%sense is under consideration) on the $D^n$-bundle arising from an
%element $\kappa \in \pi_q \BDiff_\partial (D^n)$, (respectively
%$\kappa \in \pi_q \BDiff_\partial (D^n)\otimes \Q$), determines
%an element $\bar{\phi} \in \pi_q \calM_\partial^{\pos_*}(D^n)_{h_0}$,
%(respectively $\bar{\phi} \in \pi_q {\mathcal M}^+(D^n)_{h_0}\otimes
%\Q$), with image $\xi_\ast(\kappa)$ under the map induced by the
%inclusion ${\mathcal M}_\partial^+(D^n)_{h_0} \to {\mathcal
%  M}_\partial(D^n)_{h_0}=\text{BDiff}_\partial (D^n)$.
%\bigskip\bigskip
%
%%%%%%%%%%

\section{Hatcher bundles: topology and geometry}\label{Hatcher}
Hatcher (sphere) bundles are certain smooth $S^n$-bundles with base
$S^i$, which correspond in some sense to elements in the kernel of the
classical $J$-homomorphism. A key feature of Hatcher bundles is that
they are in general non-trivial in a smooth sense, but topologically
trivial. For an introduction to these objects, see \cite[Section
  3.1]{BWW} and the references therein, especially \cite{G} and
\cite{I}.

Hatcher bundles have played an important role in casting light on the
topology of the observer moduli space of Riemannian metrics with
positive scalar curvature (see for example \cite{BHSW}), and positive
Ricci curvature (see \cite{BWW}). Although we will not consider the
observer moduli space in the current paper, it will be important for
us to recall the related notion of an {\it observer diffeomorphism
  group}.  For a manifold $N$ and a choice of basepoint $x_0 \in N$,
let $\text{Diff}_{x_0}(N)$ denote the space of diffeomorphims $N \to
N$ which fix the basepoint and for which the derivative at $x_0$ is
the identity mapping. This is the observer diffeomorphism group of
$N$. The theorem below appears as \cite[Theorem 11]{BWW}, and
summarises results from \cite{G} and \cite{I}:
\begin{theorem}\label{Hatcher_sphere}
For any $q \in \N$, there exists $N(q) \in \N$, with $N \to \infty$ as
$q \to \infty$, such that for all odd $n>N$, every element of the
group $\pi_{4q}\text{BDiff}_{x_0}(S^n) \otimes \Q$ is represented by a
Hatcher sphere bundle.
\end{theorem}
\begin{Convention}
In the sequel we will refer to a pair $q,n$ as {\it admissible} if either $n$ is odd and $n>N(q)$, or $n$ is even and $n+1>N(q)$. 
\end{Convention}
Although in \cite{BWW} we consider Hatcher sphere bundles, in the
present paper we need to consider Hatcher {\it disc} bundles. The
Hatcher sphere bundles constructed in \cite{BWW} are actually doubles
of disc bundles, and these latter bundles are the Hatcher disc
bundles. Like the corresponding sphere bundles, Hatcher disc bundles
are topologically trivial, but smoothly non-trivial in
general. However the boundary sphere bundle {\it is} trivial. The
following lemma is established as part of the main
construction in \cite{BWW}.
\begin{lemma}\label{Ric_Hatcher_disc}
Every Hatcher disc bundle admits a fibrewise metric of positive Ricci
curvature, such that all fibres are isometric and have strictly convex
boundaries.
\end{lemma}
Positive Ricci curvature of course implies $Sc_k>0$ for all $1 \le k
\le n.$ The convexity of each fibre boundary is not quite enough for
our later constructions: we will actually require a round boundary
with principal curvatures that are sufficiently large in some sense,
though of course we only require this in the context of $Sc_k>0$ for
$k \ge 2.$ The next result fixes this issue by adding a suitable
trivial ``collar'' bundle to the Hatcher disc bundle.
\begin{lemma}\label{collar}
Given constants $R>0$ and $\alpha \in [0,1/R),$
  every Hatcher disc bundle 
  admits a fibrewise metric with $Sc_2>0$,
  such that all fibres are isometric with round boundary of radius
  $R$, and the principal curvatures at the boundary (with respect to
  the inner normal) are all equal to 
  $\alpha.$ Moreover, we can arrange that an annular neighbourhood of
  each fibre boundary is isometric to an annulus in some round sphere.
\end{lemma}
\begin{proof}
The strategy will be to glue a collar to each fibre, so that the
resulting ``extended'' bundle has the properties listed in the statement
of the Lemma.

We begin by forming the fibrewise Riemannian double of the disc
bundle. Since the metric is convex at the boundary, the $C^0$-metric
which results from this doubling can be smoothed in an arbitrarily
small neighbourhood of the join to give a global fibrewise Ricci
positive metric on the resulting Hatcher sphere
bundle, as is established in \cite[Theorem 10]{BWW}. This sphere
bundle admits a section: as glued the boundary of the two disc bundles
is a trivial $S^{n-1}$-bundle, we can construct a section by choosing
the `same' point in this equator for each fibre. It is known (see for
example \cite[Proposition 2.3]{WW}), that a Ricci positive metric can
be deformed in a neighbourhood of any point to be round in some
smaller neighbourhood, while globally preserving the Ricci
positivity. Since all the fibres of the Hatcher sphere bundle are
isometric, we can make such a deformation centered around each point
of the section in an identical manner for each fibre.

Next consider a (smoothly) trivial disc bundle over $S^{4q}$, and for
$\rho\ge R$ equip this with a fibrewise metric such that each fibre is
a polar cap of radius $\rho\sin^{-1}(R/\rho)$ in the sphere
$(S^n,\rho^2ds^2_n).$ Thus each fibre has positive sectional curvature
with round boundary of radius $R$, and it is easily checked that all
principal curvatures at the boundary are equal to
$\sqrt{\rho^2-R^2}/\rho R.$ This function is increasing with $\rho$,
and therefore takes all values in the interval $[0,1/R)$ as $\rho$
  ranges over $[R,\infty)$. Thus given $\alpha$ as in the statement of
    the Lemma, we can find $\rho$ such that the principal curvatures
    at the boundary are all equal to $\alpha$.

The final step is to form a fibrewise connected sum between the
Hatcher sphere bundle and the trivial disc bundle above, where $\rho$
has been chosen so as to realize the desired principal curvatures at
the boundary. To do this, we remove identical small discs from the
latter bundle around the centre of each disc, then smoothly attach
identical tubes with $Sc_2>0$ to the resulting boundaries, as in the
proof of \cite[Theorem 2.1]{Wo}. A neighbourhood of the free end of
each attached tube is metrically a cylinder of some small fixed
radius. We then perform a similar construction around the section
points of the former bundle, again in an identical manner for each
fibre. It is straightforward to arrange that the radii of the two free
ends agree, and hence these ends can be joined to create a smooth
`extended' bundle with fibrewise $Sc_2>0$ and the boundary conditions
specified in the Lemma. This completes the construction
of the metric with $Sc_2>0$.
\end{proof}
The key topological point of importance for us concerning Hatcher
sphere bundles is the following:
\begin{proposition}\label{Hatcher_disc}
Given $q \ge 1$ and $n$ odd, consider the group
$\pi_{4q}\BDiff_\partial(D^n) \otimes \Q.$ Then provided $q,n$ are
admissible, any element of this group is represented by a Hatcher disc
bundle.
\end{proposition}
\begin{proof}
As noted above, every element of $\pi_{4q}\text{BDiff}_{x_0}(S^n)
\otimes \Q$ (for appropriate $q,n$) is represented by a Hatcher sphere
bundle which is the double of a Hatcher disc bundle.

We have a ``doubling map'' $\pi_{4q}\BDiff_\partial(D^n) \to
\pi_{4q}\BDiff_{x_0}(S^n)$, obtained by doubling the disc bundles
represented by elements of $\pi_{4q}\BDiff_\partial(D^n),$ with the
basepoint $x_0$ chosen on the common boundary between the two
discs. It is easily checked that this map is a homomorphism. This then
induces a rational map $\pi_{4q}\BDiff_\partial(D^n)\otimes \Q \to
\pi_{4q}\BDiff_{x_0}(S^n)\otimes \Q$. As each element of the target
group in the rational case is represented by a Hatcher bundle, the
rational homomorphism must be surjective. In particular, each
pre-image under this map must represent a Hatcher disc bundle, hence
the result.
\end{proof}

%%%%%%%%%%

\section{Spaces of concordances}\label{concordance}
The ultimate aim of this section is to present a proof of the Main
Theorem. This will follow by considering a diagram
\begin{center}
\begin{tikzcd}
  \pi_{4q}\calM^{\pos_*}_\partial(M \times I)_{\tilde g}\otimes \Q \arrow[d] \arrow[r] &
  \pi_{4q}\calM^{\pos_*}_\sqcup(M \times I)_g\otimes \Q \arrow[d] \\
\pi_{4q}\text{BDiff}_\partial(M \times I)\otimes \Q \arrow[r] & \pi_{4q}\text{BDiff}_\sqcup(M \times I)\otimes \Q,
\end{tikzcd}
\end{center}
where all maps are induced by inclusion.

Our strategy is this: first we show that the group in the bottom left
corner of this diagram is non-zero, and then we show that these
non-zero elements can be lifted to the group in the upper left. The
next task is to show that the map at the bottom is non-zero, from
which point the main theorem can be easily proved by a diagram-chasing
argument.

\subsection{Showing $\pi_{4q}\text{BDiff}_\partial(M \times I) \otimes \Q \neq 0.$}
Recall from Proposition \ref{Hatcher_disc} that every element of
$\pi_{4q}\text{BDiff}_\partial (D^n) \otimes \Q$ (for suitable $q,n$)
can be represented by a Hatcher disc bundle. This is a certain
$D^n$-bundle over $S^{4q}$ which is trivial when restricted to the
boundary.

For a compact manifold $N^n$, we can view an element
$\kappa \in \pi_{4q}\BDiff(N^n)$ as being
represented by a bundle $\pi: E\to S^{4q}$ where the fibres are embedded
copies of $N$ in $\R^\infty$. In other words, we
  assume the bundle $\pi: E\to S^{4q}$ is a part of the commutative
  diagram:
\begin{center}
\begin{tikzcd}
  E \arrow{d}[swap]{\pi} \arrow[hookrightarrow]{r} &
  S^{4q}\times\R^{\infty}\arrow{dl}{pr}
  \\
  S^{4q}
\end{tikzcd}
\end{center}
where $pr: S^{4q}\times\R^{\infty}\to S^{4q}$ is the
  projection on the first factor and $\pi^{-1}(x)=N_x\subset
  \R^{\infty}$ is an embedded fibre.

For such a manifold $N$ we can make the following construction. Start
by fixing an embedding $\chi: N \rh \R^\infty$, and
form the product bundle $S^{4q} \times N$, with $N$
identified with the image of $\chi$. Remove a product disc bundle to
leave $\bigl(S^{4q} \times N \bigr) \setminus \bigl(S^{4q} \times
D^n\bigr),$ where $D^n \subset N$ is a fixed embedding. In the case
where $\partial N \neq \emptyset$ we assume that $D^n$ is embedded
into the interior of $N$. Given any element $\kappa \in
\pi_{4q}\BDiff_\partial(D^n)\otimes \Q$, consider the elements
$\bar{\kappa}_i\in \pi_{4q}\BDiff_\partial (D^n)$
which map onto $\kappa$ after tensoring with
$\Q$. Any two such elements
  $\bar{\kappa}_i$, $i=1,2$, differ by a torsion element of
  $\pi_{4q}\BDiff_\partial(D^n)$ in the stable range.  Now consider
the embedded disc bundles corresponding to $\bar{\kappa}_i$, and glue
these into $\bigl(S^{4q} \times N \bigr) \setminus \bigl(S^{4q} \times
D^n\bigr)$ in the obvious way to create new $N$-bundles over
$S^{4q}$. This then determines elements of $\pi_{4q}\text{BDiff}(N)$,
and it is clear that any two such elements arising in this way differ
by torsion. Thus tensoring with $\Q$ yields a well-defined element of
$\sigma(\kappa) \in \pi_{4q}\text{BDiff}(N)\otimes \Q.$ In particular,
we can represent $\kappa$ by a Hatcher disc bundle, and therefore can
view $\sigma(\kappa)$ as being represented by the $N$-bundle over
$S^{4q}$ determined by the Hatcher bundle.

It is implicit in the above construction that the gluing can take
place without creating self-intersections. This is a valid assumption
as we can always arrange for the embedding of $N$ into $\R^\infty$ to
guarantee this, irrespective of the Hatcher bundle being used. To see
this, note that we can adjust the embedding $\chi$ by introducing an
extra dimension, and utilising the function $\psi$ from the proof of
Lemma \ref{contractible}. Specifically, we nudge all but the (inner)
boundary of $N \setminus D^n$ towards the extra dimension by setting
$\bar{\chi}:N\setminus D^n \to \R^{\infty+1}$,
$\bar{\chi}(x):=\bigl(\chi(x),\psi(\text{dist}(x,\partial)\epsilon)\bigr),$
where $\partial$ here denotes the (inner) boundary of $N \setminus
D^n$, and $\epsilon>0$ is small enough for $N \setminus D^n$ to
contain an $\epsilon$-tubular neighbourhood of the (inner)
boundary. This clearly has the desired effect.
\begin{theorem}\label{injective}
Suppose that $q,n$ are admissible, with $n$ odd, $n \ge 11$, and $4q
\le n-4$. Then for a compact manifold $N^n$, the map
$\sigma:\pi_{4q}\BDiff_\partial (D^n)\otimes \Q \to
\pi_{4q}\BDiff(N)\otimes \Q$ is injective.
\end{theorem}
Before proving Theorem \ref{injective}, we recall the following
result from \cite{Ebert}. Given an embedded disc $D^n \subset N$ as in
the above theorem, let $\mu:\text{Diff}_\partial (D^n) \to
\text{Diff}_\partial(N)$ be the homomorphism obtained by extending
diffeomorphisms of the disc by the identity across $N$.
\begin{theorem}{\rm \cite[Theorem 1.4]{Ebert}} \label{Eb}
For every odd dimensional compact manifold $N^n$ with $n \ge 11$, the
induced map
\begin{equation*}
  (\text{B}\mu)_\ast:\pi_k\BDiff_\partial
  (D^n)\otimes \Q \to \pi_k\BDiff_\partial (N)\otimes \Q
\end{equation*}
is injective when $k \neq 1$ and $k \le n-4$. Furthermore, the
  same is true when the above map is composed with the map induced by
  the inclusion $\BDiff_\partial(N) \rh \BDiff(N)$.
\end{theorem}
\begin{remark}
  The definition of $\text{Diff}_\partial (D^n)$ used in \cite{Ebert}
  is slightly different from the convention adopted in this paper:
  Ebert defines this to be the group of diffeomorphisms of the disc
  which agree with the identity in some neighbourhood of the
  boundary. This discrepancy does not matter, as the inclusion of the
  latter group into ours is a weak homotopy equivalence. This can be
  seen, for example, via an easy modification of the proof of
  \cite[Proposition 1.3]{I}.
\end{remark}
\begin{proof}[Proof of Theorem \ref{injective}]
Let us first consider the map $\text{B}\mu$. Given a group
homomorphism $\theta:G \to H$, in order to define the induced map of
classifying spaces $\text{B}\theta$, we take the universal bundle
$\text{E}G\to \text{B}G$ and form the associated bundle $\text{E}G
\times_\theta H.$ We can view this associated bundle as a principal
$H$-bundle over $\text{B}G$, and as such, it is determined by a
classifying map $\text{B}\theta:\text{B}G \to \text{B}H$.

In the specific case of the map $ \text{B}\mu :
  \BDiff_\partial (D^n) \to \BDiff_\partial (N)$, we start with the
universal $\Diff_\partial (D^n)$-bundle $\EDiff_\partial (D^n) \to
\BDiff_\partial (D^n),$ and form the associated bundle
\begin{equation*}
  \EDiff_\partial (D^n) \times_\mu \Diff_\partial
  (N) \to \BDiff_\partial (D^n).
\end{equation*}
A given point $p\in \BDiff_\partial (D^n)$
represents the image $\text{im}(\phi_p)$ of some
embedding $\phi_p: D^n \rh \R^\infty$, satisfying certain boundary
conditions as explained in Section \ref{models}. The corresponding
fibre of the bundle $\EDiff_\partial (D^n)\to
  \BDiff_\partial (D^n)$ is then the group of self-diffeomorphisms of
  this embedded disc (respecting the boundary conditions). Thus the
  corresponding fibre of the associated bundle above is the group of
  self-diffeomorphisms of the manifold $(N\setminus
  D^n)\cup_\partial \text{im}(\phi_p(D^n))$ which agree with the
  identity outside $\text{im}(\phi_p(D^n))$, and where $N\setminus
  D^n$ is embedded in a fixed way into $\R^\infty$.

On the other hand, we have a universal
$\text{Diff}_\partial(N)$-bundle $\text{EDiff}_\partial(N) \to
\text{BDiff}_\partial(N)$, where a point in the base can be identified
with the image of an embedding of $N$ into $\R^\infty$, and the
corresponding fibre is the group of diffeomorphisms of this embedded
copy of $N$. Given the definition of $B\theta$ above, we have that
$\EDiff_\partial (D^n) \times_\mu \Diff_\partial (N)$ is the bundle
that results from pulling back the universal bundle
$\EDiff_\partial(N)$ via $\text{B}\mu$, i.e., the following diagram of
fibre bundles commutes:
\begin{center}
\begin{tikzcd}
  \EDiff_\partial (D^n)\times_\mu \Diff_\partial (N)\arrow{d}
  \arrow{r}{(\mathrm{B}\mu)^*}& \EDiff_\partial (N)\arrow{d}
\\
 \BDiff_\partial (D^n)\arrow{r}{\mathrm{B}\mu}&  \BDiff_\partial (N)
\end{tikzcd}
\end{center}
Thus $\text{B}\mu$ must be the
map which takes $p \in \BDiff_\partial (D^n)$ to the point in
$\BDiff_\partial(N)$ which represents the embedding
\begin{equation*}
  (N\setminus D^n)\cup_\partial
\text{im}(\phi_p(D^n)) \rh \R^\infty\ .
\end{equation*}
Now suppose we are given an element $\kappa \in
\pi_{4q}\BDiff_\partial(D^n)\otimes \Q$, as in the statement of
the theorem. This element can be represented by an $S^{4q}$-indexed
family of discs embedded into in $\R^\infty$ (with appropriate
boundary conditions), which constitute a Hatcher disc bundle. If we
compose the map $S^{4q} \to \BDiff_\partial(D^n)$ which
determines the Hatcher bundle with the map $\text{B}\mu$ described
above, the resulting element of $\pi_{4q}\BDiff_\partial(N)$ is
clearly represented by the $S^{4q}$-indexed family of embedded copies
of $N$ in $\R^\infty$ which are formed from a standard copy of
$N\setminus D^n$ by gluing in, for given $z \in S^{4q}$, the embedded
copy of $D^n$ corresponding to $z$ in the Hatcher bundle. It is now
evident that $(\text{B}\mu)_\ast(\kappa)$ agrees with
$\sigma(\kappa).$ The theorem now follows immediately from Theorem
\ref{Eb}.
\end{proof}
As a corollary, we obtain
\begin{corollary}\label{injection}
Suppose that $q,n$ are admissible, with $n$ even, $n \ge 10$, and $4q
\le n-3$. Then for a compact manifold $M^n$, the
homomorphism
\begin{equation*}
\sigma:\pi_{4q}\BDiff_\partial (D^{n+1})\otimes \Q \to
\pi_{4q}\BDiff_\partial(M \times I)\otimes \Q
\end{equation*}
is injective.
\end{corollary}
\begin{proof}
In the case where $M$ is closed, the result is immediate from Theorem
\ref{injective}.

If $\partial M \neq \emptyset$, let $N$ be the double of $M$, so $N$
is a closed manifold. Then the inclusion of spaces $D^n
\hookrightarrow M \times I \hookrightarrow N \times I$ gives rise to a
commutative diagram of spaces
\begin{center}
\begin{tikzcd}
 & \Diff_\partial (M \times I) \arrow[rd]  & \\
  \Diff_\partial (D^n) \arrow{ru}{\mu'} \arrow{rr}{\mu} &  &
  \Diff_\partial (N \times I) 
\end{tikzcd}
\end{center}
and in turn to a commutative diagram of homotopy groups
\begin{center}
\begin{tikzcd}
 & \pi_{4q}\BDiff_\partial (M \times I) \arrow[rd]  & \\
  \pi_{4q}\BDiff_\partial (D^n) \arrow{ru}{(\text{B}\mu')_\ast}
  \arrow{rr}{(\text{B}\mu)_\ast} &  & \pi_{4q}\BDiff_\partial (N \times I).
\end{tikzcd}
\end{center}
Now by Theorem \ref{Eb} the lower map $(\text{B}\mu)_\ast$ is an
injection. Hence by commutativity, we deduce that
$(\text{B}\mu')_\ast:\pi_{4q}\text{BDiff}_\partial (D^n) \to
\pi_{4q}\text{BDiff}_\partial (M \times I)$ is also injective. The
rest of the argument now goes through exactly as in the proof of
Theorem \ref{injective}.
\end{proof}
We next combine Corollary \ref{injection} with the following result,
which was originally due to Farrell-Hsiang \cite{FH}, with an improved
range of $i$ and $n$ values provided by \cite[Theorem A]{KRW}:
\begin{theorem}\label{Farrell}
Suppose that $n \ge 11$ is odd. Then for $i \le n-4$ we
have
\begin{equation*}
  \pi_i\BDiff_\partial (D^n) \otimes \Q \cong
\begin{cases}
\Q \quad \text{ if }i \equiv 0 \text{ mod }4; \\
0 \quad \, \text{ otherwise.} \\
\end{cases}
\end{equation*}
\end{theorem}
\begin{corollary}\label{M_nonzero_downstairs}
  Suppose that $q,n$ are admissible, with $n$ even, $n \ge 10$, and
  $4q \le n-3$. Then for a compact manifold $M$ we have
\begin{equation*}
    \pi_{4q}\BDiff_\partial(M \times I) \otimes \Q \neq 0.
\end{equation*}     
\end{corollary}

%%%%%

\subsection{Showing $\pi_{4q}\calM^{\pos_*}_\partial(M \times I)_g\neq 0$}\label{m+}
Recall that $\calM^{\pos_k}_\partial(M \times I)_{\tilde g}$ denotes
the moduli space of metrics $\bar g$ on $M\times I$ with $Sc_{k}(\bar
g)>0$, satisfying certain boundary conditions which imply that $\tilde
g:=\bar g|_{\del (M\times I)}$ has $Sc_{k-1}(\tilde g)>0$. In the case
of scalar curvature, i.e., when $k=n+1=\dim (M\times I)$, we require
both $\bar g$ and $\tilde g$ to have positive scalar curvature.

Our next task is to show that non-zero elements in
$\pi_{4q}\text{BDiff}_\partial(M \times I) \otimes \Q$ in the image of
$\sigma$ can be lifted to non-zero elements in
$\pi_{4q}\calM^{\pos_*}_\partial (M \times I)_{\tilde g} \otimes \Q$. The
existence of such lifts will play a crucial role in the proof of the
main theorem.
  
\begin{proposition}\label{M-lift}
With $M$, $\dim M =n$, and $q$ as in Corollary
\ref{M_nonzero_downstairs}, assume that $M$ admits a metric with
$Sc_{k-1}>0$ for some $2\le k \le n+1$, and let
$\calM^{\pos_*}_{\partial}(M \times I)_{\tilde g}$ denote
the moduli space of metrics with $Sc_k>0$ on $M\times I$ if $k\le n$, or
the space of positive scalar curvature metrics if
$k=n+1$. Then %the group
$\pi_{4q}\calM^{\pos_*}_\partial (M \times I)_{\tilde g} \otimes \Q\neq 0$.
% is non-zero.
\end{proposition}
\begin{proof}
We have a commutative square as follows:
\begin{center}
\begin{tikzcd}
  \pi_{4q}\calM^{\pos_*}_\partial (D^{n+1})_{h_{st}}\otimes \Q \arrow{r}{\tilde{\sigma}} \arrow{d}  & \pi_{4q}\calM^{\pos_*}_\partial (M^n \times I)_{\tilde g}\otimes \Q \arrow{d}
  \\
\pi_{4q}\BDiff_\partial (D^{n+1})\otimes \Q  \arrow{r}{\sigma} & \pi_{4q}\BDiff_\partial(M^n \times I)\otimes \Q.
\end{tikzcd}
\end{center}
Here, the vertical maps are induced by inclusion, as explained above
in Section \ref{lifting}. The upper horizontal map $\tilde{\sigma}$ is
defined as follows. Fix a metric with $Sc_{k-1}>0$ on $M$. We can
deform the product metric on $M \times I$, which has $Sc_k>0$ if $k\le
n$ or positive scalar curvature if $k=n+1$, in an
$\epsilon$-neighbourhood about any interior basepoint to a round
metric of curvature 1 say, preserving the global curvature
condition. After fixing a basepoint $x_0 \in M$, let us choose the
point $(x_0,1/2) \in M \times I$ to be the basepoint of $M \times I$
and thus the centrepoint of the $\epsilon$-disc $D(\epsilon)$. Let
$\bar g$ denote the resulting deformed product
metric on $M \times I$, and as usual we set $\tilde g:= \bar
  g|_{\del(M\times I)}$ and $g:=\bar{g}|_{M \times \{0\}}$. Let $h_{st}$ denote
the (round) metric on the boundary of $D(\epsilon)$. Observe that
$h_{st}=\sin^2(\epsilon)ds^2_n$, i.e. the boundary created by removing
$D(\epsilon)$ from $(M \times I,g)$ has radius $\sin(\epsilon).$
Moreover the prinicpal curvatures at the boundary of $D(\epsilon)$ are
all equal to $\cot(\epsilon).$

By Proposition \ref{Hatcher_disc} each element of
$\pi_{4q}\BDiff_\partial (D^{n+1})\otimes \Q$ can be represented
by a Hatcher disc bundle. By Lemma \ref{collar}, setting
$R=\sin(\epsilon)$, each element of $\pi_{4q}\BDiff_\partial
(D^{n+1})\otimes \Q$ can be represented by a Hatcher disc bundle
equipped with a fibrewise $Sc_2>0$ metric, with boundary metric
$h_{st}$, and all principal curvatures at the
boundary equal to any given $\alpha \in [0,1/\sin(\epsilon))$. In
  particular, by Lemma \ref{moduli_lifts}, each element of
  $\pi_{4q}\BDiff_\partial (D^{n+1})\otimes \Q$ can be lifted to 
  $\pi_{4q}\calM^{\pos_*}_\partial (D^{n+1})_{h_{st}}\otimes \Q$

Now consider the product $\bigl(S^{4q} \times (M \times I)\bigr)
\setminus \bigl(S^{4q} \times D(\epsilon)\bigr)$, where $D(\epsilon)$
denotes the round $\epsilon$-disc, viewed as a trivial bundle over
$S^{4q}$. Equip each fibre with the restriction of $g$, to give a
fibrewise $Sc_k>0$ metric. By \cite[Theorem A]{RW}, we can smoothly
glue the Hatcher disc bundle and its fibrewise $Sc_2>0$ metric into
this product, provided that the principal curvatures $\alpha$ at the
boundary of the Hatcher bundle are greater than those at the boundary
of $D(\epsilon)$, i.e. $\alpha>\cot(\epsilon)$. We can clearly choose
such an $\alpha$ provided that $1/\sin(\epsilon)>\cot(\epsilon),$
which is trivially true.

Thus we obtain a fibrewise metric  with $Sc_k>0$ on
the glued bundle, which is of course precisely the bundle produced in
the construction of the representative for $\sigma(\kappa)$. Given the
fibrewise metric satisfies $Sc_k>0$,
%\textcolor{blue}{$Sc_k(\bar g)>0$},
%of the metric $\bar g$ ,
by Lemma \ref{moduli_lifts}, the element
$\sigma(\kappa)$ can be lifted to an element in the group 
$\pi_{4q}\calM^{\pos_*}_\partial (M^n \times I)_{\tilde{g}} \otimes \Q,$ which we
denote $\tilde{\sigma}(\kappa)$. This then defines the map
$\tilde{\sigma}$, and the commutativity of the diagram is immediate by
construction.

By Theorem \ref{injection}, we know that $\sigma$ is injective. It now
follows immediately from the non-triviality of
$\pi_{4q}\BDiff_\partial (D^{n+1})\otimes \Q$, the existence of lifts
of elements on the left-hand side of the square, and the commutativity
of the diagram, that the group $\pi_{4q}\calM^{\pos_*}_\partial (M
\times I)_{\tilde{g}} \otimes \Q$ must be non-zero as
claimed. It remains to note that the construction
  does not depend on the choice of the boundary metric $\tilde g$.
\end{proof}

%%%%%

\subsection{The map $\pi_{4q}\BDiff_\partial(M \times I)\otimes \Q \to
  \pi_{4q}\BDiff_\sqcup(M \times I) \otimes \Q$} For a compact
manifold $M$, let us define a map $\Delta:\BDiff_\sqcup(M\times
I) \to \BDiff_\partial (M \times I)$ as follows.

We begin by constructing a preliminary
map
\begin{equation*}
  \lambda_\sqcup: \Emb_\sqcup(M \times I,\R^{\infty}) \to
  \Emb(M \times I,\R^{\infty+1}).
\end{equation*}
Let $\bar e : M\times I \rh \R^{\infty}$ be the fixed
  embedding from \S2, which we will take as a ``base point'' in $\Emb_\sqcup(M \times
  I,\R^{\infty})$. Then for any embedding $\phi: M\times I \rh
  \R^{\infty}$, the map $\lambda_\sqcup$ is defined by
$\lambda_\sqcup(\phi)(x,t)=(\phi(x,t),t).$ Let $\bar{\bar e}$ denote
the image of $\bar{e}$ under this map, and observe that the elements
in $\text{im}(\lambda_\sqcup)$ agree up to infinite order with
$\bar{\bar e}$ at the lower boundary of $M \times I$, and also along
the sides $\partial M \times I$ in the case where $\partial M \neq
\emptyset$.

Notice that $\Diff_\sqcup(M \times I)$ also acts freely on
$\Emb(M \times I,\R^{\infty+1}),$ and
moreover $\lambda_\sqcup$ is equivariant with respect to the action of
this group on both domain and target spaces.

It follows trivially from Lemma \ref{contractible} that
$\text{im}(\lambda_\sqcup)$ is a contractible space
endowed with free action of $\Diff_\sqcup(M \times
  I)$. Thus we obtain a new model for $\BDiff_\sqcup(M \times
I)$:
\begin{equation*}
\BDiff_\sqcup(M \times I)=\text{im}(\lambda_\sqcup)/\Diff_\sqcup(M \times I).
\end{equation*}
We similarly have a map $\lambda_\partial:\Emb_\partial(M \times
I,\R^{\infty}) \to \Emb(M \times I,\R^{\infty+1})$ given by
$\lambda_\partial(\phi)(x,t)=(\phi(x,t),t),$ and so we obtain a new
model for the classifying space $\BDiff_\partial(M \times
I):$
\begin{equation*}
  \BDiff_\partial(M \times
  I)=\text{im}(\lambda_\partial)/\Diff_\sqcup(M \times I).
\end{equation*}
We will assume these models below.

Composing $\lambda_\sqcup$ with the quotient map
gives
\begin{equation*}
  \Emb_\sqcup(M \times I,\R^{\infty}) \to
\text{im}(\lambda_\sqcup)/\Diff_\sqcup(M \times
I)=\BDiff_\sqcup(M \times I),
\end{equation*}
and by the equivariance of
$\lambda_\sqcup$ this descends to a map
\begin{equation*}
  \Lambda: \Emb_\sqcup(M
\times I,\R^{\infty})/\Diff_\sqcup(M \times I) \to
\text{im}(\lambda_\sqcup)/\Diff_\sqcup(M \times I),
\end{equation*}
or more simply, bearing in mind our original embedding model for
$\BDiff_\sqcup(M \times I):$
\begin{equation*}
  \Lambda: \BDiff_\sqcup(M
  \times I) \to \Diff_\sqcup(M \times I).
\end{equation*}
Next, consider an element $x \in \text{im}(\Lambda) \subset
\text{BDiff}_\sqcup(M \times I)$. According to our last model for this
classifying space, $x$ can be represented by the {\it image} of an
embedding $\theta_x$ of $M \times I$ in $\R^{\infty+1}$, for which the
$\infty$-jet agrees with that of $\bar{\bar{e}}$ when restricted to $M
\times \{0\}$, and also along $\partial M \times I$ if $\partial M
\neq \emptyset.$

The next step, intuitively, is to take two copies of $x$ and translate
one copy through $\R^{\infty+1}$, (unfixing the lower boundary of
course), in such a way that we can glue the non-standard end of the
translated copy to the non-standard end of the original submanifold
representing $x$. After shrinking in the $I$ direction, the resulting
submanifold of $\R^{\infty+1}$ is again an embedded copy of $M \times
I,$ but now the {\it whole} of its boundary is standard, that is,
coincides up to infinite order with the image of the boundary of $M
\times I$ under the map $\bar{\bar{e}}$.

In more technical detail, this process of inversion and gluing amounts
to constructing a map ${\mathcal I}: \text{im}(\Lambda) \to
\text{BDiff}_\partial(M \times I)$ as follows:
\begin{equation*}
{\mathcal I}(\theta_x)(x,t)=
\begin{cases}
\bigl(\theta_x(m,2t),t \bigr) \,\,\,\, \qquad \text{ if } t\in [0,1/2] \\
\bigl(\theta_x(m,2t-1),t \bigr) \quad \text{ if } t\in [1/2,1]. \\
\end{cases}
\end{equation*}
It follows automatically from this construction that ${\mathcal
  I}(\theta_x)$ is standard on all boundary components. Note further
that since the $\infty$-jet of $\theta_x$ agrees with a product
diffeomorphism at the upper boundary, the map ${\mathcal I}(\theta_x)$
is smooth at $t=1/2$, and hence smooth globally.

Finally, set $\Delta:={\mathcal I} \circ \Lambda.$ Clearly, this gives
a map $\Delta: \BDiff_\sqcup(M \times I) \to
\BDiff_\partial (M \times I).$
\begin{proposition}\label{doubling}
With $M^n$ and $q$ as in Corollary \ref{M_nonzero_downstairs}, the
composition
\begin{equation*}
\pi_{4q}\BDiff_\partial(M \times I)
\hookrightarrow \pi_{4q}\BDiff_\sqcup(M \times I)
\xrightarrow{\Delta_*} \pi_{4q}\BDiff_\partial (M \times I)
\end{equation*}
is non-zero.
\end{proposition}
\begin{proof}
Consider the standard copy of $M \times I$ described above, and remove
a small disc $D^{n+1}$ centred on a point in $m \in M \times \{1/2\}.$
Form the product $S^{4q} \times \bigl((M \times I)\setminus
D^{n+1}\bigr).$

Take any element $0 \neq \varphi \in
\pi_{4q}\BDiff_{x_0}(S^{n+1})$, and consider the corresponding
Hatcher (sphere) bundle. This is the double of a Hatcher disc bundle
$E_\kappa$ corresponding to an element
\begin{equation*}
  \kappa \in
  \pi_{4q}\BDiff_\partial (D^{n+1})\otimes \Q.
\end{equation*}
Clearly $\kappa \neq 0$. We can assume that the fibres of $E_\kappa$
are embedded in $\R^{\infty+1}$, with each boundary sphere coinciding
with the boundary of the small disc $D^{n+1} \subset M \times I$
chosen above. We can therefore perform a fibrewise gluing of
$E_\kappa$ into the trivial bundle $S^{4q} \times
\bigl((M \times I)\setminus D^{n+1}\bigr)$ in the obvious way,
smoothing any corners as necessary. The resulting object is an $(M
\times I)$-bundle $\mathcal B$ over $S^{4q}$, which is trivial when
restricted to each boundary component. Globally $\mathcal B$ is a
trivial bundle in the topological category, but like the Hatcher
bundle, is not necessarily trivial in the smooth category. This bundle
is represented by $\sigma(\kappa) \in \pi_{4q} \BDiff_\partial
(M \times I)\otimes \Q,$ and by Corollary \ref{injection} we have that
$\sigma(\kappa) \neq 0$. By inclusion we can view $\sigma(\kappa)$ as
an element of $\pi_{4q}\BDiff_\sqcup(M \times I)\otimes \Q.$

Next, we apply the map $\Delta:\BDiff_\sqcup(M \times I) \to
\BDiff_\partial (M \times I)$ defined above simultaneously to all the
fibres in $\mathcal B$ to create a new $(M \times I)$-bundle
${\mathcal B}\cup\bar{\mathcal B}$ over $S^{4q}.$ Each fibre of this
new bundle is a standard copy of $M \times I$ from which two small
discs centered on $(m,1/4)$ and $(m,3/4)$ have been removed, and
replaced by two copies of the corresponding fibre from the Hatcher
disc bundle $E_\kappa$. Notice that the upper copy of the Hatcher disc
fibre has been inverted, in the precise sense that the $I$-direction
of the glued-in disc has been reversed. As the boundary of ${\mathcal
  B}\cup\bar{\mathcal B}$ is standard, it is clear that this bundle
represents a well-defined element $\beta \in \pi_{4q}\BDiff_\partial
(M \times I)\otimes \Q$. Since $\beta=\Delta_*(\sigma(\kappa)),$ to
prove the Proposition we simply need to show that $\beta$ is non-zero.

In order to analyze ${\mathcal B}\cup\bar{\mathcal B}$, we choose a
larger disc in each fibre $M \times I$, (say within the region $M
\times [1/8,7/8]$), which contains both glued Hatcher discs in its
interior. As the boundary of this new disc lies in the trivial part of
the bundle, we can choose the ``same'' disc for all fibres. The
collection these discs forms a bundle $\mathcal D$ over $S^{4q},$
which represents an element $\gamma
\in\pi_{4q}\BDiff_\partial(D^{n+1})\otimes \Q.$

By Corollary \ref{injection}, if $\gamma \neq 0$ then $\beta \neq 0.$
Thus the proof of the Proposition reduces to studying the bundle
$\mathcal D$.

Next, we form a sphere bundle $\mathcal S$ by gluing $\mathcal D$
(which has trivial boundary) to a trivial disc bundle over $S^{4q}.$
It is easy to see that $\mathcal S$ represents an element $\zeta \in
\pi_{4q}\BDiff_{x_0}(S^{n+1})\otimes \Q$, where $x_0$ is a fixed
choice of point on the common boundary between the two discs. (By the
boundary triviality assumptions, $x_0$ is unambiguously defined in
every fibre.) Mapping $\gamma$ to $\zeta$ gives a map
$\pi_{4q}\BDiff_\partial(D^{n+1})\otimes \Q \to
\pi_{4q}\BDiff_{x_0}(S^{n+1})\otimes \Q.$ It is easily checked
that this map is a homomorphism, so the Proposition is proved if we
can show that $\zeta \neq 0.$

Fixing the lower Hatcher discs in $\mathcal D$, we now simultaneously
slide the upper (inverted) Hatcher discs upwards in the $I$-direction,
and across the boundary into the trivial disc bundle. We do this
following the obvious continuation of the $I$-direction in each fibre,
until the upper Hatcher discs all lie in the interior of the trivial
disc bundle. Observe that this deformation re-inverts the Hatcher
discs, and so the resulting object (which as a bundle is clearly
diffeomorphic to $\mathcal S$), is a union of two new disc
bundles. Each of these disc bundles is formed from a trivial disc
bundle by the removal of an interior disc from every fibre, and gluing
in {\it the same} Hatcher disc bundle (representing the element
$\kappa\in \pi_{4q}\BDiff_\partial (D^{n+1})\otimes \Q$) into
the resulting bundle of ``holes''. But in each case this is smoothly
equivalent to the Hatcher disc bundle itself, and it follows
immediately that $\mathcal S$ is simply the Hatcher sphere bundle we
started out with. Thus $\zeta=\varphi \neq 0$. This completes the
proof of the Proposition.
\end{proof}
\begin{corollary}\label{doubling_corollary}
With $M^n$ and $q$ as in Corollary \ref{M_nonzero_downstairs}, the
homomorphism
\begin{equation*}
  \tau:\pi_{4q}\BDiff_\partial(M \times I)\otimes \Q
  \to \pi_{4q}\BDiff_\sqcup(M \times I)\otimes \Q
\end{equation*}
induced by the inclusion of the classifying spaces is non-zero. In
particular, $\tau$ is non-zero when restricted to the image of the map
$\sigma:\pi_{4q}\BDiff_\partial(D^{n+1})\otimes \Q \to
\pi_{4q}\BDiff_\partial(M \times I)\otimes \Q.$
\end{corollary}
\begin{proof}
Under the assumptions on $n$ and $q$, Theorems \ref{injection} and
\ref{Farrell} show that the image of $\sigma$ is non-zero, and hence
the first statement follows trivially from the second.

The second statement follows from the fact that given $0 \neq \kappa
\in \pi_{4q}\BDiff_\partial(D^{n+1})\otimes \Q$ as in the proof
of the above Proposition, $\sigma(\kappa)$ maps to a non-zero element
in $\pi_{4q}\BDiff_\partial(M \times I)\otimes \Q$ via the
inclusion map to $\pi_{4q}\BDiff_\sqcup(M \times I)\otimes \Q$
composed with the map $\Delta_\ast$. Hence the image of the inclusion
map in $\pi_{4q}\BDiff_\sqcup(M \times I)\otimes \Q$ is also
non-zero, as required.
\end{proof}

%%%%%

\subsection{Proof of the Main Theorem}
\begin{proof}
We have a commutative diagram
\begin{center}
\begin{tikzcd}
  \pi_{4q}\calM^{\pos_*}_\partial(M \times I)_{\tilde g}\otimes \Q \arrow[d] \arrow[r]
  & \pi_{4q}\calM^{\pos_*}_\sqcup(M \times I)_g\otimes \Q \arrow[d]
  \\
  \pi_{4q}\BDiff_\partial(M \times I)\otimes \Q \arrow[r]
  & \pi_{4q}\BDiff_\sqcup(M \times I)\otimes \Q
\end{tikzcd}
\end{center}
where all maps are induced by inclusion. By Corollary
\ref{doubling_corollary} the lower map is non-zero on the image of
$\sigma$. Moreover by Proposition \ref{M-lift} we know that each
non-zero element
$\text{im}(\sigma) \in
\pi_{4q}\BDiff_\partial(M \times I)\otimes \Q$ lifts to a
non-zero element in $\pi_{4q}\calM^{\pos_*}_\partial(M \times
I)_{\tilde g}\otimes \Q$. It follows immediately by commutativity that all such
lifts map to a non-zero element in $\pi_{4q}\calM^{\pos_*}_\sqcup(M
\times I)_g\otimes \Q$, demonstrating that the latter group is non-zero
as claimed.
\end{proof}

%%%%%%%%%%

\bigskip%\bigskip

\noindent{\it Boris Botvinnik, Department of Mathematics, University
  of Oregon, U.S.A.
  \\ Email: botvinn@uoregon.edu}

\bigskip

\noindent {\it David Wraith, 
	Department of Mathematics and Statistics, 
	National University of Ireland Maynooth, Maynooth, 
	County Kildare, 
	Ireland. \\
	Email: david.wraith@mu.ie.}

\end{document}